\RequirePackage[l2tabu, orthodox]{nag}
\documentclass[11pt]{amsart} 

\author{Andy Hammerlindl}
\address{School of Mathematical Sciences, Monash University, Victoria $3800$ Australia} \urladdr{ http://users.monash.edu.au/~ahammerl/}  \email{andy.hammerlindl@monash.edu}
\author{Yi Shi}
\address{School of Mathematical Sciences, Peking University, Beijing $100871$ China}
\email{shiyi@math.pku.edu.cn}
\title{Accessibility of derived-from-Anosov systems}
\date{\today}

\usepackage{color,xcolor}
\usepackage{amsmath}
\usepackage{amssymb}
\usepackage{amsfonts}
\usepackage{amsthm}
\usepackage{amscd}
\usepackage{graphicx}
\usepackage{setspace}
\usepackage{enumitem}
\usepackage{cleveref}
\usepackage{multicol}
\usepackage{fourier}

\usepackage[T1]{fontenc}

\makeatletter
\def\saveenum{\xdef\@savedenum{\the\c@enumi\relax}}
\def\resetenum{\global\c@enumi\@savedenum}
\makeatother

    \newcommand{\gra}{\operatorname{graph}}

    \newcommand{\bbR}{\mathbb{R}}

    \newcommand{\bbZ}{\mathbb{Z}}

    \newcommand{\bbT}{\mathbb{T}}

    \newcommand{\subof}{\subset}
    
    \newcommand{\sans}{\setminus}

    \newcommand{\Es}{E^s}
    \newcommand{\Ec}{E^c}
    \newcommand{\Eu}{E^u}
    \newcommand{\Ecu}{E^{cu}}
    \newcommand{\Ecs}{E^{cs}}

    \newcommand{\inv}{^{-1}}

    \newcommand{\invn}{^{-n}}

    \newcommand{\cF}{\mathcal{F}}
    \newcommand{\Fs}{\mathcal{F}^s}
    \newcommand{\Fc}{\mathcal{F}^c}
    \newcommand{\Fu}{\mathcal{F}^u}
    \newcommand{\Fcu}{\mathcal{F}^{cu}}
    \newcommand{\Fcs}{\mathcal{F}^{cs}}
    \newcommand{\Fus}{\mathcal{F}^{us}}
    \newcommand{\Fsig}{\mathcal{F}^{\sigma}}
    \newcommand{\cA}{\mathcal{A}}
    \newcommand{\As}{\mathcal{A}^s}
    \newcommand{\Ac}{\mathcal{A}^c}
    \newcommand{\Au}{\mathcal{A}^u}
    \newcommand{\Acu}{\mathcal{A}^{cu}}
    \newcommand{\Acs}{\mathcal{A}^{cs}}
    \newcommand{\Aus}{\mathcal{A}^{us}}
    \newcommand{\Asig}{\mathcal{A}^{\sigma}}

    \newcommand{\Ls}{L^{s}}

    \newcommand{\ep}{\epsilon}
    \newcommand{\lam}{\lambda}
    \newcommand{\Lam}{\Lambda}
    
    \newcommand{\Gam}{\Gamma}
    \newcommand{\gam}{\gamma}

    \newcommand{\al}{\alpha}

    \newcommand{\Per}{\operatorname{Per}}

    \newcommand{\delc}{\partial^c}

\numberwithin{equation}{section}
\newtheorem{thm}[equation]{Theorem}
\newtheorem{cor}[equation]{Corollary}
\newtheorem{lemma}[equation]{Lemma}
\newtheorem{claim}[equation]{Claim}
\newtheorem{prop}[equation]{Proposition}
\newtheorem{question}[equation]{\textbf{Question}}

\theoremstyle{remark}
\newtheorem*{remark} {\textbf{Remark}}

\begin{document}

\maketitle

%

\begin{abstract}
    This paper shows any
    non-Anosov
    partially hyperbolic diffeomorphism
    on the 3-torus which is homotopic to Anosov
    must be accessible. 
\end{abstract}



\section{Introduction} \label{sec:intro}

Let $M$ be a closed Riemannian manifold, and ${\rm Diff}^r(M)$ be the space which consists of all $C^r$-diffeomorphisms $(r\geq 1)$ of $M$ and is endowed with $C^r$-topology. We say $f\in{\rm Diff}^r(M)$ is partially hyperbolic if there exist a continuous $Df$-invariant splitting $TM=E^s\oplus E^c\oplus E^u$ and two continuous functions $\sigma,\mu:M\rightarrow\mathbb{R}$, such that $0<\sigma<1<\mu$ and
$$
\|Df(v^s)\|<\sigma(x)<\|Df(v^c)\|<\mu(x)<\|Df(v^u)\|
$$
for every $x\in M$ and unit vector $v^*\in E^*(x)$, for $*=s,c,u$. Let ${\rm PH}^r(M)$ denote the set consisting of all partially hyperbolic diffeomorphisms on $M$. It is obvious that ${\rm PH}^r(M)$ is an open set in ${\rm Diff}^r(M)$ with respect to the $C^r$-topology.

Robust transitivity is an important hallmark of chaotic dynamics. We say $f\in{\rm Diff}^1(M)$ is \emph{robustly transitive} if $f$ admits a $C^1$-neighborhood $\mathcal{U}$ such that every $g\in\mathcal{U}$ is robustly transitive. Due to the structural stability of Anosov systems, a transitive Anosov diffeomorphism is robustly transitive.  In fact, for a long time, Anosov diffeomorphisms were the only known examples of robustly transitive diffeomorphisms until the discovery of counterexamples by M.~Shub \cite[Page 39]{Shub} and R.~Ma\~n\'e \cite{Mane78}.

R. Ma\~n\'e constructed a non-hyperbolic robustly transitive diffeomorphism on $\bbT^3$.
His example, now known as Ma\~n\'e's example, was partially hyperbolic and homotopic to an Anosov automorphism. We say $f\in{\rm PH}^r(\bbT^3)$ is a \emph{derived-from-Anosov diffeomorphism} or a \emph{DA-diffeomorphism}, if it is homotopic to an Anosov automorphism. Here the Anosov automorphism is given by the linear part $f_*:\pi_1(\bbT^3)\rightarrow\pi_1(\bbT^3)$ of $f$.

Partially hyperbolic dynamics in dimension 3 and particularly on the 3-torus is now very well
understood. Derived-from-Anosov diffeomorphisms have been studied extensively.
For instance, every DA-diffeom\-orphism is dynamically coherent and leaf conjugate to its linear part \cite{BBI2,Hammerlindl,Po-jmd}.
There exists an open set of conservative DA-diffeom\-orphisms whose
center Lyapunov exponents have a different sign than their linear part
\cite{PT}, and
for such examples the center foliation is ``minimal yet measurable''
\cite{PTV14}.
The disintegration
of measure along the center foliation has been studied
in detail \cite{VY17}.
Every conservative partially hyperbolic DA-diffeomorph\-ism
is ergodic \cite{HamUres, gs20XXrigidity}.
In certain settings, it is further known to be Bernoulli
\cite{PTV18}.

However, a major question of transitivity for derived-from-Anosov diffeomorphisms remains open.
\begin{question}\label{ques}
	Is every partially hyperbolic derived-from-Anosov diffeomorphism on $\bbT^3$ transitive?
\end{question}
Since the DA-diffeomorphisms form an open subset of ${\rm PH}^1(\bbT^3)$, if every DA-diffeomorphism is transitive, then they are all robustly transitive as well.

For partially hyperbolic diffeomorphisms, one of the major tools for establishing transitivity is due to M. Brin \cite{brin1975trans}, showing that if every point is non-wandering and the system is accessible, then it is transitive. Here, we say $f\in{\rm PH}^r(M)$ is \emph{accessible} if for every $x,y\in M$, there exists a piecewise smooth curve from $x$ to $y$, such that each smooth piece is contained in a leaf of either the stable or unstable foliation of $f$. 

In light of Brin's result \cite{brin1975trans}, one approach to construct a non-transitive example of DA-diffeomorphisms might be to start with a linear Anosov automorphism, which is not accessible as it has an invariant $us$-foliation, and to do a pitchfork bifurcation on a periodic point producing a non-Anosov example. One might also hope that one could do this deformation in such a way that the
deformation "blows air" into the leaves of the original $us$-lamination. This DA-diffeomorphism would then not be accessible and in fact would have an invariant $us$-lamination. However in this paper, we prove a result showing that such an approach is impossible.

\begin{thm} \label{thm:main}
	Let $f : \bbT^3 \to \bbT^3$
	be a $C^{1+\al}$ partially hyperbolic derived-from-Anosov diffeomorphism.
	If $f$ is not Anosov, then it is accessible.
\end{thm}

This theorem shows that any version of constructing Ma\~n\'e's example \cite{Mane78} must be accessible.
Moreover, since every Anosov diffeomorphism on $\bbT^3$ is transitive, Brin's work \cite{brin1975trans} has the following corollary.

\begin{cor}\label{cor:accessibility}
	Let $f : \bbT^3 \to \bbT^3$
	be a $C^{1+\al}$ partially hyperbolic derived-from-Anosov diffeomorphism.
    If the non-wandering set of $f$ is $\bbT^3$, then $f$ is transitive.
\end{cor}

Following this corollary, Question \ref{ques} is reduced to showing that the non-wand\-ering set of a DA-diffeomorphism is the whole torus $\bbT^3$. 
Such a question remains challenging. 
For instance, a well-known open problem is whether the time-one maps of geodesic flows on closed surfaces with constant negative curvature must be robustly transitive \cite{wilkinson2010conservative}. These maps are partially hyperbolic and stably accessible. However, it is unknown if it holds robustly that the non-wandering set is the whole manifold \cite{bonatti-guelman}.
In the specific case of derived-from-Anosov diffeomorphisms on $\bbT^3$,
see \cite{Po-jdde} for a detailed discussion of transitivity
and the non-wandering set.

Applying Theorem 5.1 of  \cite{gs20XXrigidity}, Theorem \ref{thm:main} has the following corollary. 

\begin{cor}
	Let $f : \bbT^3 \to \bbT^3$
	be a $C^{1+\al}$ partially hyperbolic derived-from-Anosov diffeomorphism, then either $f$ is accessible, or
	it is Anosov and every periodic point of $f$ has the same center Lyapunov exponent with its linear part $f_*$.
\end{cor}


Another major motivation for Theorem \ref{thm:main} is the study of accessibility in its own right.
M.~Grayson, C.~Pugh, and M.~Shub \cite{grayson1994stably} first introduced the concept of accessibility in the volume-preserving setting
as a tool for establishing ergodicity. However, accessibility has many applications in the non-volume-preserving setting such as the result of Brin \cite{brin1975trans} mentioned above. C.~Pugh and M.~Shub conjectured that accessibility is an open and dense property among partially hyperbolic diffeomorphisms, volume preserving or not.
This property has been established in a number of settings, including the case of one-dimensional center \cite{BHHU} (which always holds when the diffeomorphism is defined on a three-dimensional manifold).

F.~Rodriguez Hertz, J.~Rodriguez Hertz, and R.~Ures \cite{HHU-JMD-08} conjectured that the existence of a torus tangent to the $us$-direction is the unique obstruction to accessibility of partially hyperbolic diffeomorphisms in dimension 3.
Theorem \ref{thm:main} handles the last unknown case for accesssibility on the 3-torus. In fact, when combined with results in \cite{Ha17}, it gives a complete description of all possible accessibility classes for any 3-manifold with solvable fundamental group.

\begin{cor}\label{cor:solv}
	Let $f$ be a partially hyperbolic diffeomorphism on a 3-manifold with solvable fundamental group,
	then exactly one of the following holds:
	\begin{enumerate}
		\item $f$ is accessible;
		\item $f$ has a minimal invariant $us$-foliation and is an Anosov diffeomorphism on $\bbT^3$;
		\item $f$ has a 2-torus tangent to $E^s\oplus E^u$.
		
	\end{enumerate}
\end{cor}

See \cite{Ha17} for further details on the accessibility classes in the case of the third item above.
This corollary gives a positive answer in the case of solvable fundamental group
to the following question which still remains open for general 3-manifolds.

\begin{question}
    Suppose f is a partially hyperbolic diffeomorphism with one-dimen\-sional
    center and 
	$\emptyset\neq \Gamma(f)\neq M$ is an $f$-invariant lamination tangent to $E^s\oplus E^u$.
    Does $\Gamma(f)$ have a compact leaf?
\end{question}

See Section 4 of the survey paper \cite{RHRHU-survey}
(and Problem 22 in particular)
for further details.

\smallskip

The proof of Theorem \ref{thm:main} extends techniques first developed
in \cite{HamUres} and \cite{gs20XXrigidity} to prove ergodicity
for DA-diffeomorphisms in the volume-preserving setting.
However, there is a key difference in the overall approach to the proof.
For any DA-diffeomorphism, there is a semiconjugacy, discovered by Franks
\cite{Franks1}, from the non-linear system to its linear part.
One of the first steps in \cite{HamUres} is to use the volume-preserving
assumption to show that this semiconjugacy is in fact a true conjugacy;
that is, a homeomorphism.
In the setting of the current paper, we do not assume the system is volume
preserving and so we must always handle the possibility that the semiconjugacy
is not bijective. This significantly complicates a number of parts of the
proof. For instance, see \cref{prop:curves} and its use in the proof of
\cref{prop:stos} below.

\vskip2mm

\noindent{\bf Organization of the paper:} In Section \ref{sec:univ}, we study the semiconjugacy and the lift of invariant foliations to the universal cover $\bbR^3$. In Section \ref{sec:lami}, we prove a series of properties of the minimal invariant $us$-laminations on $\bbT^3$. In Section \ref{sec:rigidity}, we prove the main theorem.

\section{Laminations on the universal cover} \label{sec:univ} 

This section analyzes the dynamics when lifted to the universal cover
$\bbR^3$. Later sections will use the results proved here
in order to analyse the original system on the 3-torus.
Before looking at the partially hyperbolic system,
we first establish a property for curves on $\bbR^2$
which will be of use later.

\begin{prop} \label{prop:curves}
    Let $S$ be a collection of curves in $\bbR^2$
    and $X$ be a dense subset of $\bbR^2$ with
    the following properties:
    \begin{enumerate}
        \item each curve in $S$ is the graph
        $\gra(g) = \{ (x,g(x)) : x \in \bbR \}$
        of a continuous function $g : \bbR \to \bbR;$
        \item
        no two curves topologically cross; and
        \item
        if $\gam$ is a curve in $S$ and $(x_0,y_0) \in X,$
        then the translation $\gam + (x_0,y_0)$
        is also a curve in $S.$
    \end{enumerate}
    Then the curves in $S$ are straight lines
    and all have the same slope.
\end{prop}
\begin{remark}
    If $\gra(g_1)$, $\gra(g_2)$ are 
    curves in $S,$ 
    then the condition of no topological crossings
    implies that either
    $g_1(x) \le g_2(x)$ holds for all $x \in \bbR$
    or
    $g_2(x) \le g_1(x)$ holds for all $x \in \bbR.$
\end{remark}

\begin{proof}
    Consider the closure of $S$ in the compact-open topology.
    That is, the graph of a function $g : \bbR \to \bbR$
    belongs to $\bar S$
    if and only if there is a sequence $\{\gra(g_n)\}$ in $S$
    such that 
    $g_n|_K$ converges uniformly to $g|_K$
    on every compact subset $K$ of $\bbR.$
    One can show that the no crossing condition on $S$
    implies a no crossing condition on $\bar S.$
    If $\gam_n \to \gam$ in the compact-open topology
    and $(x_0, y_0) \in X,$
    then the translates $\gam_n + (x_0, y_0)$
    converge to $\gam + (x_0, y_0)$
    in the compact-open topology.
    Hence, $\bar S$ is invariant under all translates
    $(x_0, y_0) \in X.$
    For $\gam \in \bar S$ and any point $(x_0, y_0) \in \bbR^2,$
    let $(x_n, y_n)$ be a sequence in $X$ converging to $(x_0, y_0).$
    Then $\gam + (x_n, y_n)$ converges to $\gam + (x_0, y_0)$
    in the compact-open topology
    and so $\gam + (x_0, y_0) \in \bar S.$

    Now knowing that
    $\bar S$ is invariant under \emph{all} translations in $\bbR^2,$
    we can show it is linear.
    Indeed,
    let $\gra(g)$ be a curve in $\bar S$
    passing though the origin so that $g(0) = 0.$
    For $a, b \in \bbR$,
    define functions $g_n : \bbR \to \bbR$ by
    $g_n(x) = g(x - a) + g(a) + \tfrac{1}{n}.$
    Then
    $g_n(a) > g(a)$ and the no crossing condition imply that
    \[
        g(a) + g(b) + \tfrac{1}{n} = g_n(a + b) \ge g(a + b)
    \]
    for all $n.$
    A similar argument shows that
    $g(a) + g(b) - \tfrac{1}{n} \le g(a + b)$
    for all $n,$
    and so $g(a) + g(b) = g(a + b).$
    As $g$ is continuous and additive, it is linear.
    The translates of $\gra(g)$ produce a linear foliation
    on all of $\bbR^2$ that no other curve of $\bar S$ can cross.
    This implies that every curve of $\bar S$ is linear
    and of the same slope.
\end{proof}

With \cref{prop:curves} established, we now consider
the dynamics.
Let $f : \bbT^3 \to \bbT^3$ be partially hyperbolic and homotopic
to an Anosov diffeomorphism.
Lift $f$ to a map on the universal cover.
All of the analysis in this section will be on $\bbR^3,$
and so, by a slight abuse of notation,
from now until the end of the section
we only use $f$ to denote the lifted map.
This lift $f : \bbR^3 \to \bbR^3$ is at finite distance
from a hyperbolic linear map $A : \bbR^3 \to \bbR^3.$ Here $A=f_*$ is the linear part of $f$.
Up to replacing $f$ by its inverse, we may assume the center direction
of $A$ is contracting.
That is, the logarithms of the eigenvalues of $A$ satisfy
\[
    \lam^s(A) < \lam^c(A) < 0 < \lam^u(A).
\]
Many properties have been established for $f,$
first in the absolutely partially hyperbolic case
\cite{BI, BBI2, ham-thesis}
and then extended to the case of pointwise partial
hyperbolicity \cite{hp2014pointwise}.

There are unique invariant foliations
tangent to the bundles
$\Eu$, $\Es$, $\Ecs$, $\Ecu$, and $\Ec$ of $f.$
Denote these foliations by
$\Fu$, $\Fs$, $\Fcs$, $\Fcu$, and $\Fc$.
For the linear map $A$,
we adopt the notation used in
\cite{hp2015classification, hp2018survey}
and write
$\Au$, $\As$, $\Acs$, $\Acu$, $\Ac$, and $\Aus$
for the invariant linear foliations of $A$.
For both $f$ and $A$, all of these foliations
have quasi-isometrically embedded leaves
\cite[Theorem 3.5]{hp2014pointwise}
\cite[Proposition 2.6]{ham-thesis}.
That is,
there is $Q > 1$ such that if $x$ and $y$
lie on the same leaf of the foliation,
then
\begin{math}
    d_{\cF}(x,y) < Q \| x - y \| + Q
\end{math}
where $ \| x - y \| $ is the usual distance in $\bbR^3$
and $d_{\cF}(x,y)$ is distance measured along the leaf.
In this section, we use $d_u, d_c, d_s$
to denote distance measured along leaves
associated to the non-linear system $f.$

The foliations have \emph{global product structure}
\cite[Proposition 2.15]{ham-thesis}.
That is, for $x,y \in \bbR^3,$
the following pairs of sets intersect in a unique point:
\begin{enumerate}
    \item $\Fcs(x)$ with $\Fu(y),$
    \item
    $\Fcu(x)$ with $\Fc(y),$
    \item
    $\Fc(x)$ with $\Fu(y)$ if $x \in \Fcu(y),$ and
    \item
    $\Fc(x)$ with $\Fs(y)$ if $x \in \Fcs(y).$
\end{enumerate}
There is a \emph{semiconjugacy} \cite{Franks1},
a continuous surjective map
$h : \bbR^3 \to \bbR^3$ which satisfies
$h(f(x)) = A(h(x))$ and $h(x + z) = h(x) + z$
for all $x \in \bbR^3$ and $z \in \bbZ^3.$
Further, $h$ is a finite distance from the identity map on $\bbR^3.$

For the center, center-stable, and center-unstable
foliations,
$h$ defines a bijection on the spaces of leaves \cite[\S $3$]{ham-thesis}.
That is,
if $x \in \bbR^3$ and $v = h(x),$ then
$h(\Fc(x)) = \Ac(v)$ and $h \inv(\Ac(v)) = \Fc(x).$
Similar equalities hold for $cs$ and $cu$ in place of $c$.
The restriction $h|_{\Fc(x)} : \Fc(x) \to \Ac(v)$
is a continuous surjective map,
but in general it is not a homeomorphism.
Throughout this section,
we use the letter $v$ to denote a point in $\bbR^3$
associated to the linear dynamics of $A$
and we use $x$ and $y$ to denote points associated to the non-linear
dynamics of $f$.

We assume all of the foliations have been given
an orientation.
For points $x$ and $y$ on a one-dimensional leaf,
write $x = y$, $x < y$, or $x > y$
to denote their relative positions with respect to
this orientation.

\begin{prop} \label{prop:monotonic}
    The semiconjugacy is monotonic along center leaves.
    That is, the orientations may be chosen
    so that if $L \in \Fc$ and $h(L) \in \Ac$ is its image,
    then $x \le y$ implies $h(x) \le h(y)$
    for all points $x,y \in L.$
\end{prop}
This result is well known, but a proof does not appear
to be given anywhere in the prior literature.

\begin{proof}
    Suppose $h$ is not monotonic along a center leaf $L.$
    Then there are points $x < y < z$ along $L$
    such that $h(x) = h(z) \ne h(y).$
    As $A \inv$ expands the linear center direction,
    \begin{math}
        \| A \invn h(x) - A \invn h(y) \|
        =
        \| h f \invn(x) - h f \invn(y) \|
        \to \infty
    \end{math}
    as $n \to \infty.$
    As $h$ is a finite distance from the identity,
    it follows that
    \begin{math}
        d_c(f \invn(x), f \invn(y)) \to \infty. \end{math}
    The same analysis shows 
    \begin{math}
        d_c(f \invn(y), f \invn(z)) \to \infty \end{math}
    and since the points along the leaf
    have the ordering $f \invn(x) < f \invn(y) < f \invn(z),$
    it follows that
    \begin{math}
        d_c(f \invn(x), f \invn(z)) \to \infty \end{math}
    as well.
    Using that center leaves are quasi-isometrically embedded,
    one can show that
    \begin{math}
        \| A \invn h(x) - A \invn h(z) \|
        \to \infty
    \end{math}
    which contradicts the fact that $h(x) = h(z).$
\end{proof}
\begin{cor} \label{cor:interval}
    For each $v \in \bbR^3,$ the preimage $h \inv(v)$ consists
    either of a single point or a compact
    segment inside a center leaf.
\end{cor}
\begin{proof}
    As $h$ is surjective,
    $h \inv(v)$ is non-empty.
    Let $L \in \Fc$ be such that $h(L)$ contains $v.$
    As $h$ is a bijection on the spaces of center leaves,
    it follows that $h \inv(v) \subof L.$
    As $L$ is properly embedded and $h$ is a finite distance
    from the identity,
    $h \inv(v)$ is a compact subset of $L.$
    As $h|_L$ is monotonic, $h \inv(v)$ is connected.
\end{proof}
We now show $h$ also defines a bijection between the spaces
of unstable leaves.

\begin{prop} \label{prop:utou}
    For an unstable leaf $L \in \Fu$ of $f,$
    the image $h(L)$ is an unstable leaf of $A$,
    and $h|_L$ is a homeomorphism.
\end{prop}
\begin{proof}
    Suppose $x, y \in L.$
    Then
    \begin{math}
        \| f \invn(x) - f \invn(y) \| \to 0
    \end{math}
    as $n \to \infty.$
    As $h$ is uniformly continuous and a semiconjugacy,
    it follows that
    \begin{math}
        \| A \invn h(x) - A \invn h(y) \| \to 0
    \end{math}
    which is only possible if $h(x)$ and $h(y)$ are on the same
    linear unstable leaf.
    If $x \ne y,$ then
    \begin{math}
        \| f^n(x) - f^n(y) \| \to \infty
    \end{math}
    from which one can show that $h(x) \ne h(y).$
    As an injective proper map from one copy of $\bbR$ to another,
    $h|_L$ must be a homeomorphism.
\end{proof}
\begin{prop} \label{prop:stocs}
    If $L \in \Fs$ is a stable leaf of $f,$
    then $h(L)$ is a continuous curve embedded in a center stable leaf
    of $A$.
    In this linear center-stable leaf,
    each linear center leaf intersects $h(L)$ exactly once.
\end{prop}
\begin{proof}
    The stable leaf $L$ lies in a center-stable leaf of $f$
    and by global product structure
    $L$ intersects every center subleaf exactly once.
    As $h$ is a bijection on the spaces of center and center-stable leaves,
    the result follows.
\end{proof}
We now assume that $f$ is not accessible.
Then there is a non-empty lamination $\Gam \subof \bbR^3$
consisting of the non-open accessibility classes
\cite{RHRHU-accessibility}.
Call these the $us$-\emph{leaves} of $f.$
By global product structure,
each $us$-leaf is bifoliated
by stable and unstable leaves
and intersects every center leaf of $f$ exactly once.

Under this assumption, a stable
analogue of \cref{prop:utou} holds.

\begin{prop} \label{prop:stos}
    For a stable leaf $\Ls$ of $f,$
    the image $h(\Ls)$ is a (strong) stable leaf of $A$,
    and $h|_{\Ls}$ is a homeomorphism.
\end{prop}
\begin{proof}
    This is an adaptation of the argument
    given in \cite[\S $6$]{HamUres}.
    The key idea is to intersect the $us$-leaves
    with one fixed center-stable leaf,
    apply $h$ to the resulting collection of stable leaves
    and show that the images 
    satisfy the hypotheses of \cref{prop:curves}.
    We now give the details.

    \smallskip{}

    Consider the linear center-stable leaf $\Acs(0)$
    passing through the origin in $\bbR^3.$
    The pre-image $h \inv(\Acs(0))$ is a center-stable leaf of $f.$
    For any $us$-leaf $L \in \Gam,$
    the intersection $L \cap h \inv(\Acs(0))$
    is a stable leaf and so its image $h(L) \cap \Acs(0)$
    satisfies the conclusions of \cref{prop:stocs}.
    Define a set of curves in $\Acs(0)$ by
    \[
        S \ = \ \big \{ \ h(L) \cap \Acs(0) \ : \ L \in \Gam \ \big \}.
    \]
    Since $h$ is monotonic along center leaves,
    the curves in $S$ do not topologically cross.
    For $z \in \bbZ^3,$ define a translation $\tau_z : \Acs(0) \to \Acs(0)$
    by setting $\tau_z(v)$ to the unique intersection of
    $\Au(v + z)$ with $\Acs(0).$
    If $\gam = h(L) \cap \Acs(0)$ is a curve in $S,$
    \cref{prop:utou} shows that
    $\tau_z(\gam) = h(L + z) \cap \Acs(0)$
    is also a curve in $S.$
    As the unstable foliation of a linear Anosov map on $\bbT^3$
    is minimal \cite{Franks1},
    the set of translations
    $\{ \tau_z : z \in \bbZ^3 \}$
    is dense in the set of all rigid translations of $\Acs(0).$

    The collection of curves $S$
    satisfies the hypotheses of \cref{prop:curves}
    where $\Acs(0)$ is identified with $\bbR^2$
    and the linear center foliation is identified with 
    vertical lines on $\bbR^2.$
    All curves in $S$ are thus linear.
    Since the collection $S$ is invariant under $A$,
    these curves must be aligned with the linear stable direction.

    \smallskip{}

    We have shown that for a $us$-leaf $L$ of $f,$
    the image $h(L)$ is a $us$-leaf for the linear map $A$.
    On $\bbT^3,$
    such linear $us$-leaves are dense in $\bbT^3$
    and so on the universal cover
    the image of the closed set $\Gam$ must be $h(\Gam) = \bbR^3.$
    Consider now any stable leaf $\Ls \in \Fs.$
    Since $\Ls$ does not cross through
    any $us$-leaf of $f,$
    the monotoni\-city of $h$ implies that the
    image $h(\Ls)$
    does not topologically cross any leaf of
    the linear $us$-foliation.
    Hence,
    $h(\Ls)$ must lie in a single linear $us$-leaf.
    It also lies in a single linear $cs$-leaf
    and so it is a linear stable leaf.
    That $h|_{\Ls}$ is a homeomorphism
    is proved similarly to \cref{prop:utou}.
\end{proof}
\bigskip{}

We now consider the set $Y \subof \bbR^3$
consisting of all points where $h$ is injective.
That is,
$y \in Y$ if and only if $h \inv (h(y)) = \{y\}.$

\begin{prop} \label{prop:Yus}
    The set $Y$ is a union of $us$-leaves.
\end{prop}
\begin{proof}
    We first show that the complement of $Y$
    is $us$-saturated.
    If $x \in \bbR^3 \sans Y,$
    then $x$ lies on a compact interval
    $J = h \inv (h (x)).$
    Using stable and unstable holonomies,
    we can map $J$ to a compact interval
    on any other center leaf
    and propositions \ref{prop:utou} and \ref{prop:stos}
    show that this other interval
    is also mapped to a point by $h.$

    If $U \subof \bbR^3$ is an open accessibility class,
    then $h(U)$ is a single linear $us$-leaf.
    Any center segment contained in $U$ maps to a single point under
    $h$ and so $U$ and $Y$ are disjoint.
    This shows that $Y$ is a subset of $\Gam.$
\end{proof}
\begin{prop} \label{prop:countable}
    For any linear center leaf $\Ac(v),$
    the set $\Ac(v) \sans h(Y)$ is countable.
\end{prop}
\begin{proof}
    Let $x$ be such that $h$ maps $\Fc(x)$ to $\Ac(v).$
    Then any point of $\Ac(v) \sans h(Y)$
    is the image of an interval of positive length in $\Fc(x)$
    and there can only be countably many
    disjoint intervals of this form.
\end{proof}
\begin{prop} \label{prop:closure}
    For a point $x \in \bbR^3,$ the following are equivalent:
    \begin{enumerate}
        \item $x$ is in the closure of $Y;$
        \item
        $h|_{\Fc(x)}$ is not locally constant at $x;$
        \item
        either $x$ lies in $Y$ or $x$ is an endpoint of the interval $h(h \inv(x)).$
    \end{enumerate} \end{prop}
\begin{proof}
    Since $h|_{\Fc(x)} : \Fc(x) \to \Ac(h(x))$
    is continuous, surjective, and monotonic,
    it is straightforward to show $(2)$ $\Leftrightarrow$ $(3).$
    We show $(1)$ $\Leftrightarrow$ $(2).$
    Suppose $x \in \bar Y \sans Y.$
    As $Y$ is $us$-saturated,
    there are points $y_n \in \Fc(x) \cap Y$
    converging to $x.$
    The images $h(y_n)$ are distinct from each other
    and so $h|_{\Fc(x)}$ is not locally constant at $x.$
    Conversely, if $h|_\Fc(x)$ is not locally constant at $x,$
    then for any neighbourhood $x \in J \subof \Fc(x),$
    the image $h(J)$ has positive length
    and by \cref{prop:countable}
    there is $v \in h(J) \cap h(Y)$
    and so $h \inv(v) \in J \cap Y.$
\end{proof}
\begin{prop} \label{prop:minimalcover}
    For any $us$-leaf $L \in \Gam,$
    the closure of \
    \begin{math}
        \bigcup_{z \in \bbZ^3} (L + z)
    \end{math}
    contains $\bar Y.$
\end{prop}
\begin{remark}
    This shows that $\bar Y$ when projected
    down to a subset of $\bbT^3$ yields a minimal $us$-lamination.
    Moreover,
    this is the unique minimal $us$-lamination.
\end{remark}
\begin{proof}
    For a point $x \in \bar Y,$
    consider a short center segment $J \subof \Fc(x)$
    containing $x$ and
    such that $h(J)$ has positive length.
    As
    \begin{math}
        \bigcup_{z \in \bbZ^3} h(L + z)
    \end{math}    
    is a dense union of linear $us$-leaves,
    there is $z \in \bbZ^3$ such that $h(L + z)$
    intersects the interior of $h(J).$
    Let $v$ be the point of intersection.
    Then $h \inv(v)$ is contained in $J$
    and so $L + z$ intersects $J.$
\end{proof}
\begin{prop} \label{prop:percover}
    For any open set $U$ which intersects $\bar Y,$
    there is $x \in \bar Y \cap U, k \ge 1,$ and $z \in \bbZ^3$
    such that $f^k(x) = x + z.$
    That is, $x$ projects down to a periodic point on $\bbT^3.$
\end{prop}
\begin{proof}
    Using \cref{prop:closure}, one can show that
    $h(U)$ has non-empty interior.
    This interior contains a point $v$ which
    projects down a periodic point for
    the linear Anosov diffeomorphism on $\bbT^3;$
    that is, there are $k \ge 1$ and $z \in \bbZ^3$
    such that $A^k(v) = v + z.$
    By the leaf conjugacy,
    \begin{math}
        f^k( h \inv(v) ) + z = h \inv (v)
    \end{math}
    and as $v$ is in the interior of $h(U),$
    it follows that $h \inv(v)$ is contained in $U.$
    If $h \inv (v)$ is a singleton set, it is the desired point $x.$
    Otherwise, we may take either endpoint of the interval $h \inv(v)$
    to be $x.$
\end{proof}

\section{Minimal lamination and semiconjugacy on $\bbT^3$}\label{sec:lami}

We have now finished working on the universal cover.
All of the results from now until the end of the paper
will be for the original partially hyperbolic diffeomorphism
$f : \bbT^3 \to \bbT^3$.
In proving \cref{thm:main} we may freely lift $f$ to a finite cover
and replace it by an iterate;
therefore, we assume the invariant
bundles $E^c$, $E^u$, and $E^s$ are oriented and that
$f$ preserves these orientations.
The diffeomorphism is homotopic to a linear Anosov diffeomorphism
$A : \bbT^3 \to \bbT^3$
and the logarithms of the eigenvalues of $A$ satisfy
\[
    \lam^s(A) < \lam^c(A) < 0 < \lam^u(A).
\]
As before, we assume that $f$ is non-accessible.
Then there is a lamination $\Gam \subof \bbT^3$
consisting of the non-open accessibility classes of $f$.
By the work of Potrie,
the lamination $\Gam$ contains a unique minimal sublamination \cite{pot2014few}. 
This can also be seen directly from \cref{prop:minimalcover} above.
Let
$\Fu$, $\Fs$, $\Fcs$, $\Fcu$, and $\Fc$
denote the invariant foliations of $f$
considered as foliations defined on $\bbT^3$,
and similarly let
$\Au$, $\As$, $\Acs$, $\Acu$, $\Ac$, and $\Aus$
denote the invariant linear foliations of the toral automorphism $A$.
We fix an orientation of $\Fc$ and of $\Ac$.
Then each center leaf $\Fc(q)$ splits into two half-leaves
$$
\Fc(q)\setminus\{q\}=
\Fc_+(q)\cup\Fc_-(q),
$$
where $+$ and $-$ are determined by the orientation of $\Fc$ in $\bbT^3$.
For every $y\in\Fc_+(x)$, we let $[x,y]^c$ and $(x,y)^c$ denote the closed and open segments contained in $\Fc(x)$ with endpoints $x$ and $y$ respectively.

Let $h : \bbT^3 \to \bbT^3$ be the Franks semiconjugacy,
now considered as a continuous surjective map on the 3-torus.
That is, $h$ is isotopic to the identity and
$h(f(x)) = A(h(x))$
for all $x \in \bbT^3$.
\Cref{cor:interval} implies an analogous result for the semiconjugacy on the
3-torus: for each $v \in \bbT^3$,
the preimage $h \inv(v)$ consists
either of a single point or a compact
segment inside a center leaf.

If $x \in \bbT^3$ is such that $h\inv(h(x))=\{x\}$,
we define $x_+ = x_- = x$.
If instead $h\inv(h(x))$ is a positive length interval,
we define $x_+$ and $x_-$ to be the two endpoints of $h\inv(h(x))$ with corresponding orientation. 
That is, $x_+\in\Fc_+(x_-)$ and $x_-\in\Fc_-(x_+)$.
Define
\[ \Lam=\bigcup_{x \in \bbT^3}\{x_+,x_-\}. \]

%
\begin{prop}\label{prop:su-lamination}
	The set $\Lam$ and semiconjugacy $h$ satisfy the following properties:
	\begin{enumerate}
		\item The set $\Lam$ is a $us$-saturated minimal set, i.e. if $x\in\Lam$, then
		$$
		\Fus(x)\subset\Lam \qquad {\it and} \qquad
		\overline{\Fus(x)}=\Lam.
		$$ 
		
		\item The periodic points of $f|_\Lam$ are dense in $\Lam$.
		\item For every $x\in\Lam$, if $x=x_+$, then $y=y_+$ for every $y\in\Fus(x)\subset\Lam$; if $x=x_-$, then $y=y_-$ for every $y\in\Fus(x)\subset\Lam$.
		
		\item For $\sigma=u,s,us$, the semiconjugacy $h|_{\Lam}$ maps a leaf of the $\sigma$-foliation of $f$ to a leaf of the $\sigma$-foliation of $A$: 
		$$
		h(\Fsig(x))=\Asig(h(x)), \qquad\forall x\in\Lam.
		$$ 
		Moreover, $h:\Fsig(x)\rightarrow\Asig(h(x))$ is a homeomorphism.
	\end{enumerate}
\end{prop}

\begin{proof}
    These are all consequences of the results in \cref{sec:univ}.
    Item (1) follows from propositions \ref{prop:closure} and
    \ref{prop:minimalcover}.
    Item (2) follows from \cref{prop:percover}.
    Item (3) follows from \cref{prop:Yus}.
    Item (4) follows from propositions \ref{prop:utou} and \ref{prop:stos}
    and global product structure.
\end{proof}

\begin{prop} \label{prop:su-Holder}
	There exist constants $C_1>1$ and $0<\alpha<1$, such that for every $\sigma=u,s,us$, the homeomorphism $h|_{\Fsig(x)}:\Fsig(x)\rightarrow\Asig(h(x))$ is bi-H\"older continuous, i.e. for every $x_1,x_2\in \Fsig(x)$, we have
	$$
	d_{\Fsig}(x_1,x_2)\leq C_1\cdot d_{\Asig}(h(x_1),h(x_2))^{\alpha}, \quad
	d_{\Asig}(h(x_1),h(x_2)) \leq C_1\cdot d_{\Fsig}(x_1,x_2)^{\alpha}.
	$$	
\end{prop}

\begin{proof}
	We first prove the case $\sigma=u$ and $d_{\Fu}(x_1,x_2)\leq C\cdot d_{\Au}(h(x_1),h(x_2))^{\alpha}$. 
	There exists $\delta_0>0$, such that for every $x_1,x_2\in\Fu(x)\subset\Lam$, if $d_{\Au}(h(x_1),h(x_2))<\delta_0$, then $d_{\Fu}(x_1,x_2)<1$. Otherwise, there exist $x^n_1,x^n_2\in\Fu(x^n)\subset\Lam$ such that $d_{\Fu}(x^n_1,x^n_2)=1$ and $d_{\Au}(h(x^n_1),h(x^n_2))\rightarrow0$. Taking a subsequence if necessary, we have $x^n_1\rightarrow y_1$ and $x^n_2\rightarrow y_2$ with $y_2\in\Fu(y_1)$, $d_{\Fu}(y_1,y_2)=1$, and $h(y_1)=h(y_2)$. This contradicts the fact that $h$ is a homeomorphism on $\Fu(y_1)$.
	
	Now we assume that $d_{\Au}(h(x_1),h(x_2))\ll\delta_0$. Let $k$ be the largest positive number such that 
	$$
	d_{\Au}(A^k\circ h(x_1),A^k\circ h(x_2))<\delta_0.
	$$
	Then we have 
	$$
	d_{\Au}(h(x_1),h(x_2))\geq
	\exp\left(-(k+1)\cdot\lam^u(A)\right)\cdot\delta_0.
	$$
	
	On the other hand, from the semiconjugacy and $d_{\Au}(A^k\circ h(x_1),A^k\circ h(x_2))<\delta_0$, we have
	$d_{\Fu}(f^k(x_1),f^k(x_2))<1$. This implies 
	$$
	d_{\Fu}(x_1,x_2)<\mu^{-k},
	$$
	where $\mu=\inf_{z\in\mathbb{T}^3}m(Df|_{E^u(z)})>1$.
	
	If $\mu\geq\exp\lam^u(A)$, then we have
	$$
	d_{\Fu}(x_1,x_2)<\frac{\exp\lam^u(A)}{\delta_0}\cdot d_{\Au}(h(x_1),h(x_2)).
	$$
	Otherwise, we take $0<\alpha<1$ such that $\exp(\alpha\lam^u(A))<\mu$. Then we have
	$$
	d_{\Fu}(x_1,x_2)<\mu^{-k}<\exp\left(-k\alpha\lam^u(A)\right)
	\leq\frac{\exp(-\alpha\lam^u(A))}{\delta_0^{\alpha}}\cdot d_{\Au}(h(x_1),h(x_2))^{\alpha}.
	$$
	
	The proof of the other inequality and the case $\sigma=s$ are the same. If $h|_{\Lam}$ is bi-H\"older continuous on every leaf of $\Fs$ and $\Fu$, then it is bi-H\"older continuous on every leaf of $\Fus$. 
\end{proof}

\begin{prop}\label{prop:c-continuity}
	Let $x,y\in\Lam$ and a sequence of points $x_n\in\Fus(x)\cap\Fc_+(y)$ such that 
	$$
	x_{n+1}\in(y,x_n)^c \qquad {\it and} \qquad
	\lim_{n\rightarrow\infty}d_{\Fc}(x_n,y)=0.
	$$
	For every $z\in\Fc(x)\cap\Lam$ satisfying $h(z)\neq h(x)$, there exists $\delta_z>0$, such that if we denote $h^{us}_{x,x_n}:\Fc(x)\rightarrow\Fc(x_n)$ the holonomy map induced by $\Fus$ in $\Lam$ from $x$ to $x_n$, and $z_n=h^{us}_{x,x_n}(z)$, then
	$$
	d_{\Fc}(x_n,z_n)\geq\delta_z.
	$$
	The same conclusion holds for the sequence of points $x_n\in\Fus(x)\cap\Fc_-(y)$.
\end{prop}

\begin{proof}
    We denote $\delta_1=d_{\Ac}(h(x),h(x))>0$ and assume $x=x_+$. If $z\in\Fc_+(x)$, then $h(z)\in\Ac_+(h(x))$.
    Since $\lim_{n\rightarrow\infty}d_{\Fc}(x_n,y)=0$, we have $y=y_+$.
    There exists $w_1\in\Fc_+(y)$, such that $d_{\Ac}(h(y),h(w_1))=\delta_1$. Moreover, there exists $N_1>0$, such that
    $
    d_{\Ac}(h(x_{N_1}),h(y))\leq\delta_1/2.
    $
    This implies $w_1\in(x_n,z_n)$ and $[x_{N_1},w_1]^c\subset[x_n,z_n]^c$ for every $n\geq N_1$. So we define
    $$
    \delta_z=\min\left\lbrace d_{\Fc}(x_1,z_1),\cdots\cdots, d_{\Fc}(x_{N_1-1},z_{N_1-1}), d_{\Fc}(x_{N_1},w_1)
    \right\rbrace.
    $$
    
    If $z\in\Fc_-(x)$, then $h(z)\in\Ac_-(h(x))$. Since $\lim_{n\rightarrow\infty}d_{\Fc}(x_n,y)=0$, there exists $N_2>0$, such that $d_{\Ac}(h(x_{N_2}),h(y))\leq\delta_1/2$. This implies $y\in(z_n,x_n)^c$ and $[z_{N_2},y]^c\subset[z_n,x_n]^c$ for every $n\geq N_2$. So we define
    $$
    \delta_z=\min\left\lbrace d_{\Fc}(x_1,z_1),\cdots\cdots, d_{\Fc}(x_{N_2-1},z_{N_2-1}), d_{\Fc}(z_{N_2},y)
    \right\rbrace.
    $$
    This proves the case $x=x_+$. The proof for $x=x_-$ is the same.
\end{proof}

%
%
%
%
%
%

Define the real number
\begin{math}
    \lam^- = \inf \{ \lam^c(p) : p \in \Per(f|_\Lam) \}.
\end{math}

\begin{lemma}\label{lem:per-exp}
    If $p \in \Per(f|_\Lam)$, then $\lam^c(p)\leq0$.
	This implies $\lam^-\leq0$.
\end{lemma}
\begin{proof}
	Let $p\in\Per(f|_\Lam)$ with period $\pi$. Since $\Lam=\bar Z = \bigcap_{x \in \bbT^3} \delc h \inv(h(x))$ and $us$-saturated, there exists a sequence of points $x_n\in Z$, such that $x_n\in\Fc(p)$ and $x_n$ converge to $p$ in $\Fc(p)$. Moreover, for every $n$, there exists $k_n>0$, such that $A^{k_n}(h(x_1))$ is between $h(x_n)$ and $h(p)$ in $\Ac(h(p))$.
	From the semiconjugacy and $x_n\rightarrow p$ in $\Fc(p)$, we have 
	$$
	\lim_{k\rightarrow+\infty}d(f^{k\pi}(x_1),p)=0.
	$$
	This implies $\lam^c(p)\leq0$ for every $p \in \Per(f|_\Lam)$.
\end{proof}

\begin{prop} \label{prop:ergper}
	If $\mu$ is an ergodic measure supported on $\Lam$ and $\lam^c(\mu)<0$, then there exists a sequence of periodic points $q_n\in\Lam$, such that 
	$$
	\lim_{n\rightarrow\infty}\lam^c(q_n)=\lam^c(\mu).
	$$
    In particular, if $\mu$ is an ergodic measure supported on $\Lam,$
    then its central Lyapunov exponent satisfies $\lam^c(\mu) \ge \lam^-.$
\end{prop}
\begin{proof}
	Assume $\mu$ is an ergodic measure supported on $\Lam$ with $\lam^c(\mu)<0$.
    Take a $\mu$-typical point $x.$
    This point is recurrent and has a Pesin stable manifold.
    The shadowing lemma of Pesin theory
    leads to a periodic point $q_n$ of period $\pi(q_n)$ with center Lyapunov exponent
    satisfying 
    $$
    |\lam^c(q)-\lam^c(\mu)|<\min\left\lbrace \frac{1}{n},\frac{|\lam^c(\mu)|}{2}\right\rbrace.
    $$
    
    Moreover, the Pesin stable manifold of $q_n$ transversely intersects the unstable manifolds $\Fu(x)$ of $x$. If $y$ denotes the point of intersection, then
    $$
    y\in\Fu(x)\subset\Fus(x)\subset\Lam.
    $$  
    Since $\Lam$ is a compact invariant set, and $f^{k\pi(q_n)}(y)\rightarrow q_n$ as $k\rightarrow+\infty$, we have $q_n\in\Lam$.
\end{proof}

\begin{prop} \label{prop:lamneg}
    The inequality $\lam^- \leq\lam^c(A)<0$ holds.
\end{prop}
\begin{proof}
    We only need to show that there exists an ergodic measure $\mu$ supported on $\Lam$, such that $\lam^c(\mu)<0$. Actually, we consider the measure $\mu_0$ of maximal entropy of $f$, then 
    \cite{ures2012intrinsic} shows that its support ${\it supp}(\mu_0)\subset\Lam.$
    
   Again, \cite{ures2012intrinsic} 
    shows that $\lam^c(\mu_0)\leq\lam^c(A)<0$.   By Proposition \ref{prop:ergper} and Lemma \ref{lem:per-exp}, we prove that $\lam^-\leq\lam^c(A)<0$.
\end{proof}

\section{Rigidity of center Lyapunov exponents}\label{sec:rigidity}

In this section, we prove Proposition \ref{prop:samelyap}, which states that all periodic points in $\Lam$ have the same center Lyapunov exponent. This implies that $\Lam$ is hyperbolic.
We then use this hyperbolicity to show that $\Lam = \mathbb{T}^3$.

\begin{lemma} \label{lem:remetric}
	For every $\ep > 0$,
	up to changing the metric,
	there is a point $p \in \Per(f|_\Lam)$
	such that
	\[
	\log \| Df|_{\Ec(x)}\| > \lam^c(p)-\ep
	\]
	holds for all $x \in \Lam.$
\end{lemma}

\begin{proof}
	We have proved that $\lam^-\leq\lam^c(\mu)$ for every ergodic measure $\mu$ supported on $\Lam$. From the definition of $\lam^-$, there exists a sequence of periodic points $p_n\in{\rm Per}(f|_{\Lam})$ such that $\lim_{n\rightarrow\infty}\lam^c(p_n)=\lam^-<0$. 
	
	For every $\ep>0$, up to changing the metric, we have
	$$
	\log \| Df|_{\Ec(x)}\|>\lam^--\frac{\ep}{2}
	$$
	for all $x \in \Lam$. Then we take $p=p_n$ for $n$ large enough which proves the lemma.
\end{proof}

\begin{lemma}\label{lem:heteroclinic}
	There exist two constants $C_2>0$ and $0<\beta<1$, such that for every two periodic points $p,q\in\Lam$, there exist two sequence of points $x_n\in\Fs(p)$, $y_n\in\Fu(x_n)$ with $y_n\in\Fc(q)$, such that
	$$
	\lim_{n\rightarrow\infty}d_{\Fc}(y_n,q)=0,  
	\quad {\it and} \quad
	d_{\Fu}(x_n,y_n)\leq\frac{C_2}{D_n^{\beta}},
	\quad {\it where} \quad
	D_n=d_{\Fs}(p,x_n).
	$$        
	Moreover, for every $\eta>0$, there exists $N_{\eta}>0$, such that for every $n>N_{\eta}$,
	$$
	d_{\Fc}(f^k(y_n),f^k(q))\leq \eta, \qquad \forall k\geq 0.
	$$                                                
\end{lemma}
\begin{proof}
	Since $\Lam=\bigcup_{x \in \bbT^3}\delc h \inv(h(x))$ is a $us$-minimal set, at least one branch $\Fc_+(q)$ or $\Fc_-(q)$ contains a sequence of points $z_m\in\Lam$, such that $z_m$ converges to $q$ in $\Fc(q)$. We assume $z_m\in\Fc_+(q)\cap\Lam$ for every $m$. Then we have $h(z_m)\neq h(q)$ and $h(z_m)$ converges to $h(q)$ in $\Ac_+(h(q))$.
	
	The points $h(p)$ and $h(q)$ are periodic points $A$. Since $\As$ is an\ linear foliation with algebraic irrational rotation vector on $\bbT^3$, there exist $C_2'>0$ and two sequences of points $x_n'\in\As(h(p))$, $y_n'\in\Au(x_n')$ with $y_n'\in\Ac_+(h(q))$, such that
	$$
	d_{\Ac}(h(q),y_n')\leq\frac{C_2'}{\sqrt{D_n'}}
	\qquad {\rm and} \qquad
	d_{\Au}(y_n',x_n')\leq\frac{C_2'}{\sqrt{D_n'}},
	$$
	where $D_n'=d_{\As}(h(p),x_n')\rightarrow+\infty$ as $n\rightarrow+\infty$.
	
	If we have $p=h\inv (h(p))$, then for every $n$, $h\inv(x_n')$ and $h\inv(y_n')$ are single points, we denote
	$$
	x_n=h\inv(x_n')\in\Fs(p), \qquad {\rm and} \qquad
	y_n=h\inv(y_n')\in\Fu(x_n).
	$$
	
	Otherwise, the set $\delc h \inv(h(p))=\Lam\cap h \inv(h(p))$ consists of two periodic points, and $p$ is one of them. This implies  $h \inv(h(\Fus(p)))\cap\Lam$ contains two $us$-leaves, and one of them is $\Fus(p)$. So there exist unique a pair of points
	$$
	x_n=h\inv(x_n')\cap\Fs(p), \qquad {\rm and} \qquad
	y_n=h\inv(y_n')\cap\Fu(x_n).
	$$
	
	In both cases, we have $y_n\in\Fc_+(q)$. Moreover, for every $h(z_m)\in\Ac_+(h(q))$, there exists $N_m>0$, such that $y_n'$ is contained in the open interval with endpoints $h(q)$ and $h(z_m)$ in $\Ac_+(h(q))$ for every $n\geq N_m$. This implies $y_n$ is between $q$ and $z_m$ in $\Fc_+(q)$. Thus  $\lim_{n\rightarrow\infty}d_{\Fc}(z_m,q)=0$ implies $\lim_{n\rightarrow\infty}d_{\Fc}(y_n,q)=0$.
	
	Finally, let $H:\bbR^3\rightarrow\bbR^3$ be the lifting map of the semiconjugacy $h:\bbT^3\rightarrow\bbT^3$. Then there exists a constant $K>0$, such that $d(H,{\rm Id})<K$. For $D_n'=d_{\As}(h(p),x_n')$, we have 
	$$
	D_n=d_{\Fs}(p,x_n)>d_{\As}(h(p),x_n')-2K=D_n'-2K.
	$$
	So for $n$ large enough, we have $D_n>D_n'/2$. On the other hand, from $d_{\Au}(y_n',x_n')\leq C_2'/\sqrt{D_n'}$ and Propisition \ref{prop:su-Holder}, we have
	$$
	d_{\Fu}(x_n,y_n)\leq C_1\cdot d_{\Au}(x_n,y_n)^{\alpha}
	\leq C_1\cdot(2C_2')^{\alpha}\cdot D_n^{-\frac{\alpha}{2}}.
	$$
	Thus we set $C_2=C_1\cdot(2C_2')^{\alpha}$ and $\beta=\alpha/2$.
	
	Finally, since $y_n\in\Fus(p)\subset\Lam$ converges to $q$ in $\Fc_+(q)$, this implies for every $n$,
	$$
	\lim_{k\rightarrow\infty}d_{\Fc}(f^k(y_n),f^k(q))=0.
	$$
	So for $n=1$, there exists $K_0>0$, such that $d_{\Fc}(f^k(y_n),f^k(q))<\eta$ for every $k\geq K_0$. Let $\pi(q)$ be the period of $q$, and set $N_{\eta}>0$, where $y_{N_{\eta}}$ is contained in $(q,f^{K_0\pi(q)}(y_1))$. Then 
	for every $n>N_{\eta}$,
	$$
	d_{\Fc}(f^k(y_n),f^k(q))\leq \eta, \qquad \forall k\geq 0.
	$$     
\end{proof}

\begin{remark}
	Since $\lim_{n\rightarrow\infty}d_{\Fc}(y_n,q)=0$, we can assume $y_{n+1}\in(q,y_n)^c$ by taking subsequence. This allows us to apply Proposition \ref{prop:c-continuity}.
\end{remark}

Now we fix the constant 
$$
\theta=-\frac{1}{2}\cdot\frac{\beta\cdot\log\left( \sup_{x\in\mathbb{T}^3}\|Df|_{E^s(x)}\|\right) }{\log\left( \inf_{x\in\mathbb{T}^3}m(Df|_{E^u(x)})\right) }\in(0,1).
$$
We can state the main result of this section, which implies that $\Lam$ is a hyperbolic set and equal to $\bbT^3$. The proof of this proposition is similar to Proposition 4.1 of \cite{gs20XXrigidity}.

\begin{prop}\label{prop:samelyap}
	All periodic points in $\Lam$ have the same center Lyapunov exponent and satisfies
	$$
	\lam^c(p)\leq\lam^c(A)<0, \qquad \forall p\in\Per(f|_{\Lam}).
	$$
\end{prop}
\begin{proof}
	Lemma \ref{lem:per-exp} proved that $\lam^c(p)\leq 0$ for every $p\in \Per(f|_\Lam)$. Proposition \ref{prop:ergper} and \ref{prop:lamneg} proved that there exists a sequence of periodic points $q_n\in\Per(f|_{\Lam})$, such that $\lim_{n\rightarrow\infty}\lam^c(q_n)=\lam^-\leq\lam^c(A)<0$.
	
	Assume there exists two periodic points $r,q\in\Per(f|_{\Lam})$ such that $\lam^c(r)<\lam^c(q)\leq0$. Denote
	\begin{align}\label{delta0}
	     \delta_0=\frac{\theta}{4}\left(\lam^c(q)-\lam^c(r)\right)>0.
	\end{align} 
	From Lemma \ref{lem:remetric}, there exists a periodic point $p\in\Per(f|_{\Lam})$, such that
	\begin{align}\label{per-p}
	    \lam^c(p)\leq\lam^c(r), \qquad {\rm and} \qquad
	    \lam^c(p)<\log\|Df|_{E^c(x)}\|+\delta_0, 
	    \quad \forall x\in\Lam.
	\end{align}
	Denote $n_0$ be the minimal common period of $p$ and $q$.
	 
	Let $\eta_0>0$, such that for every $z_1,z_2\in\bbT^3$ satisfying $d(z_1,z_2)\leq 3\eta_0$, we have
	$$
	\left| \log\|Df|_{E^c(z_1)}\|-\log\|Df|_{E^c(z_2)}\|\right| <\delta_0.
	$$
	Since $p\in\Per(f|_{\Lam})$ and $\lam^c(p)<0$, there exists $z\in\Fc(p)\cap\Lam$, such that 
	$$
	h(z)\neq h(p), \quad {\rm and} \quad
	d_{\Fc}\big(f^k(p),f^k(z)\big)\leq\eta_0, \qquad \forall k\geq 0.
	$$
	We assume $x\in\Fc_+(p)$.
	
	We apply Lemma \ref{lem:heteroclinic} to $p$ and $q$, there exists $x_n\in\Fs(p)$, $y_n\in\Fu(x_n)$ with $y_n\in\Fc(q)$ such that
	$$
	\lim_{n\rightarrow\infty}d_{\Fc}(y_n,q)=0,  
	\quad {\it and} \quad
	d_{\Fu}(x_n,y_n)\leq\frac{C_2}{D_n^{\beta}},
	\quad {\it where} \quad
	D_n=d_{\Fs}(p,x_n).
	$$
	By taking subsequence, we can assume $y_{n+1}\in(q,y_n)^c$. Moreover, there exists $N_0>0$, such that for every $n>N_0$, $y_n$ satisfies
	$$
	d_{\Fc}\big(f^k(y_n),f^k(q)\big)\leq\eta_0,
	\qquad \forall k\geq 0.
	$$

	Denote $h^{us}_{p,y_n}:\Fc(p)\rightarrow\Fc(y_n)$ the holonomy map induced by $\Fus$ in $\Lam$ from $p$ to $y_n$, and $z_n=h^{us}_{p,y_n}(z)$. Proposition \ref{prop:c-continuity} shows that there exists $\delta_z>0$, such that
	$$
	d_{\Fc}(y_n,z_n)\geq\delta_z.
	$$

	\begin{figure}[htbp]
		\centering
		\includegraphics[width=10cm]{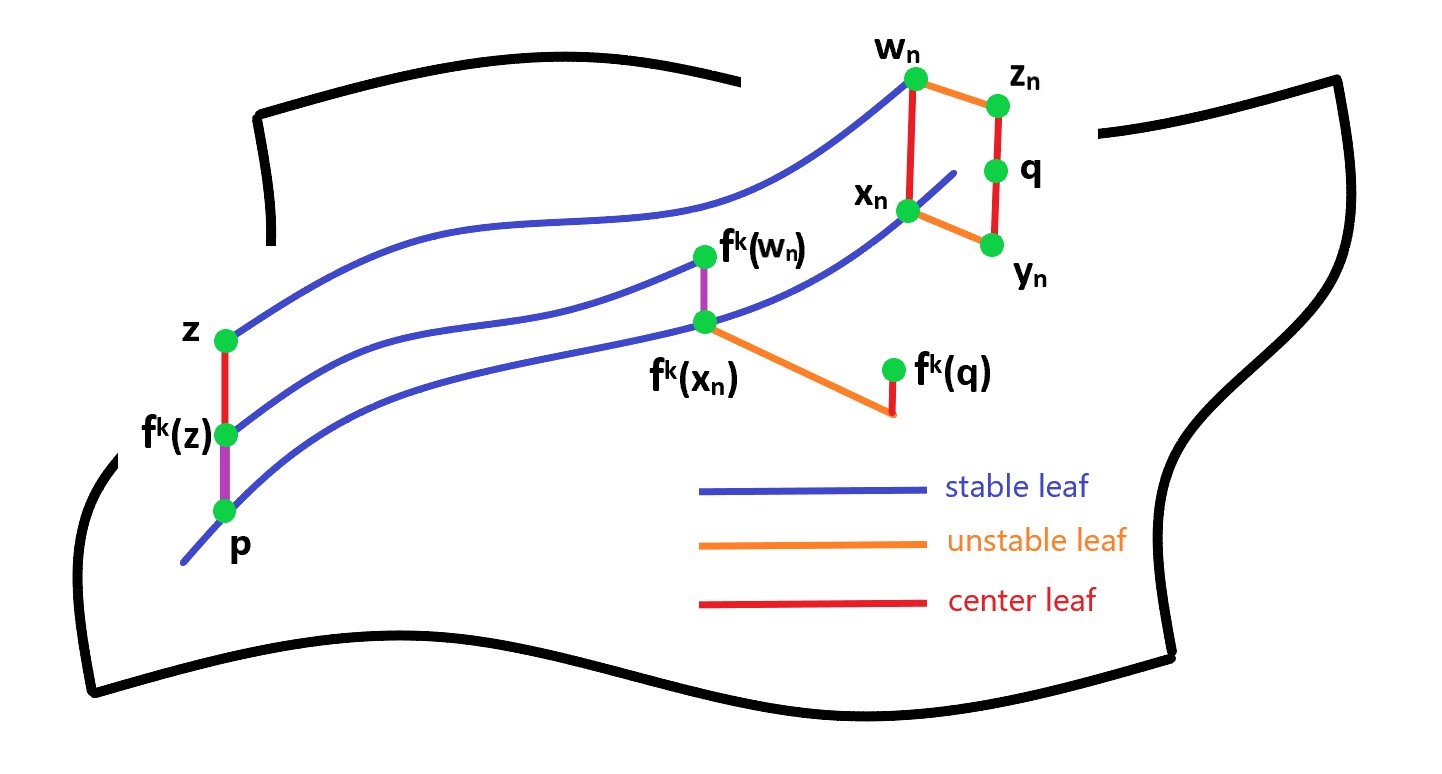}
		\caption{Holonomy maps of stable foliations}
	\end{figure}

	\begin{claim}
		If we denote $h^s_{p,x_n}:\Fc(p)\rightarrow\Fc(x_n)$ the holonomy map induced by $\Fs$ in $\Fcs(p)$ from $p$ to $x_n$, and $w_n=h^s_{p,x_n}(z)$, then there exists $N_1>0$, such that for every $n>N_1$, it satisfies
		$$
		d_{\Fc}(x_n,w_n)\geq\delta_z/2.
		$$
	\end{claim}
	
	\begin{proof}[Proof of the claim]
        Consider the holonomy map $h^u_{y_n,x_n}:\Fc(y_n)\rightarrow\Fc(x_n)$ induced by $\Fu$ in $\Fcu(q)$ for $y_n$ to $x_n$. We denote $w_n=h^u_{y_n,x_n}(x_n)$. From Theorem B of \cite{PSW}, $h^u_{y_n,x_n}$ is $C^{1+}$ continuous. Moreover, since $d_{\Fu}(x_n,y_n)\rightarrow0$ as $n\rightarrow\infty$, the derivative of $h^u_{y_n,x_n}$ converges to $1$ as $n\rightarrow\infty$. So  there exists $N_1>0$, such that for every $n>N_1$, it satisfies 
		$$
		d_{\Fc}(x_n,w_n)\geq\delta_z/2.
		$$
	\end{proof}

	For every $n$, we denote 
	\begin{itemize}
		\item $m_n$ be the smallest positive integer where $d_{\Fs}\big(p,f^{m_nn_0}(x_n)\big)\leq1$.
		\item $k_n\in[0,m_n]$ be the largest positive integer satisfying
		$$
		d_{\Fu}\big(f^{kn_0}(x_n),f^{kn_0}(y_n)\big)\leq\eta_0
		$$ 
		for every $0\leq k\leq k_n$.
	\end{itemize}
    \begin{claim}
    	There exists $N_2>0$, such that for every $n>N_2$, we have
    	$$
    	\frac{k_n}{m_n}>\theta.
    	$$
    \end{claim}
    \begin{proof}[Proof of the claim]
    	From the definition of $m_n$ and $k_n$, they satisfy
    	\begin{itemize}
    		\item $m_n$ satisfies $D_n\cdot \left(\sup_{x\in\mathbb{T}^3}\|Df|_{E^s(x)}\|\right)^{(m_n-1)n_0}> 1$, this implies
    		$$
    		m_n<-\frac{\log D_n}{n_0\cdot\log\left[ \sup_{x\in\mathbb{T}^3}\|Df|_{E^s(x)}\|\right]}+1.
    		$$
    		\item $k_n$ satisfies $(C_2/D_n^{\beta})\cdot\left(\inf_{x\in\mathbb{T}^3}m(Df|_{E^u(x)})\right)^{k_nn_0}<\eta_0$, this implies
    		$$
    		k_n> \frac{\beta\cdot\log D_n+\log\eta_0-\log C_2}{n_0\cdot\log\left[ \inf_{x\in\mathbb{T}^3}m(Df|_{E^u(x)})\right] }.
    		$$
    	\end{itemize}
    	Since $D_n\rightarrow+\infty$ as $n\rightarrow\infty$, we also have $m_n\rightarrow\infty$ as $n\rightarrow\infty$. Thus for 
    	$$
    	\theta=-\frac{1}{2}\cdot\frac{\beta\cdot\log\left( \sup_{x\in\mathbb{T}^3}\|Df|_{E^s(x)}\|\right) }{\log\left( \inf_{x\in\mathbb{T}^3}m(Df|_{E^u(x)})\right) },
    	$$
    	there exists $N_2>0$, such that if $n> N$, then
    	$k_n/m_n>\theta$.
    \end{proof}

	\begin{claim}
		For every $n>\max\{N_0,N_1,N_2\}$, we have
		$$
		d_{\Fc}\big(f^{m_nn_0}(z),p\big)\leq
		\exp\left[m_nn_0\cdot(\lam^c(p)+\delta_0)\right],
		$$
		and
		$$
		d_{\Fc}\big(f^{m_nn_0}(x_n),f^{m_nn_0}(w_n)\big)>
		\exp\left[m_nn_0\cdot\left(\lam^c(p)+2\delta_0\right) \right] 
		\cdot\min\left\lbrace \eta_0,\delta_z/2\right\rbrace.
		$$
	\end{claim}
    \begin{proof}[Proof of the claim]
    	Now we let $n>\max\{N_0,N_1,N_2\}$.
    	The length of $(m_nn_0)$-iteration of the segment $[p,z]^c$ satisfies
    	$$
    	d_{\Fc}\big(f^{m_nn_0}(z),p\big)\leq
    	\exp\left[m_nn_0\cdot(\lam^c(p)+\delta_0)\right].
    	$$
    	
    	On the other hand, since $d(x_n,q)\leq 2\eta_0$, for every $1\leq k\leq k_nn_0$, either 
    	$$
    	d_{\Fc}\big(f^{k-1}(x_n),f^{k-1}(w_n)\big)>\eta_0;
    	$$
    	or
    	$$
    	d_{\Fc}\big(f^k(x_n),f^k(w_n)\big)>
    	\|Df|_{E^c(f^{k-1}(q))}\|\cdot
    	\exp(-\delta_0)\cdot d_{\Fc}(f^{k-1}(x_n),f^{k-1}(w_n)).
    	$$
    	Recall that $d_{\Fc}(x_n,w_n)>\delta_z/2$, this implies that
    	$$
    	d_{\Fc}\big(f^{k_nn_0}(x_n),f^{k_nn_0}(w_n)\big)>
    	\exp\left[k_nn_0\cdot (\lam^c(q)-\delta_0) \right] 
    	\cdot\min\left\lbrace \eta_0,\delta_z/2\right\rbrace .
    	$$
    	
    	\vskip2mm
    	
    	If $k_n=m_n$, then the definition of $\delta_0$ (\ref{delta0}) and the fact $\lambda^c(p)\leq\lambda^c(r)$ (\ref{per-p}) imply
    	$$
    	d_{\Fc}\big(f^{m_nn_0}(x_n),f^{m_nn_0}(w_n)\big)>
    	\exp\left[m_nn_0\cdot\left(\lam^c(p)+2\delta_0\right) \right] 
    	\cdot\min\left\lbrace \eta_0,\delta_z/2\right\rbrace.
    	$$
    	We already proved the claim. 
    	
    	Otherwise, for every $k_nn_0<k\leq m_nn_0$, we have
        \begin{itemize}
        	\item either
        	$$
        	d_{\Fc}\big(f^{k-1}(x_n),f^{k-1}(w_n)\big)>\eta_0;
        	$$
        	\item or $d_{\Fc}\big(f^{k-1}(x_n),f^{k-1}(w_n)\big)
        	\leq\eta_0$, this implies
        	\begin{align*}
        	d_{\Fc}&\big(f^k(x_n),f^k(w_n)\big) \\
        	&>\|Df|_{E^c(f^{k-1}(x_n))}\|\cdot
        	\exp(-\delta_0)\cdot d_{\Fc}\big(f^{k-1}(x_n),f^{k-1}(w_n)\big)
        	\end{align*}
        	Due to the choice of periodic point $p$ in (\ref{per-p}), we have 
        	$$
        	\lam^c(p)<\log\|Df|_{E^c(f^{k-1}(x_n))}\|+\delta_0,
        	$$ 
        	thus 
        	$$
        	\qquad d_{\Fc}\big(f^k(x_n),f^k(w_n)\big)>
        	\exp\left(\lam^c(p)-2\delta_0\right) \cdot d_{\Fc}(f^{k-1}(x_n),f^{k-1}(w_n)).
        	$$
        	
        \end{itemize}
        
      Both cases imply
    	\begin{align*}
    	&d_{\Fc}\big(f^{m_nn_0}(x_n),f^{m_nn_0}(w_n)\big) \\
    	&>\exp\left[k_nn_0\cdot (\lam^c(q)-\delta_0)+ (m_n-k_n)n_0\cdot(\lam^c(p)-2\delta_0)\right] 
    	\cdot\min\left\lbrace \eta_0,\delta_z/2\right\rbrace \\
    	&>\exp\left[m_nn_0\cdot\left(\theta\lam^c(q)+(1-\theta)\lam^c(p)-2\delta_0\right) \right] 
    	\cdot\min\left\lbrace \eta_0,\delta_z/2\right\rbrace
    	\end{align*}
    	Since $\delta_0=\frac{\theta}{4}\left(\lam^c(q)-\lam^c(r)\right)$, we have
    	$$
    	d_{\Fc}\big(f^{m_nn_0}(x_n),f^{m_nn_0}(w_n)\big)>
    	\exp\left[m_nn_0\cdot\left(\lam^c(p)+2\delta_0\right) \right] 
    	\cdot\min\left\lbrace \eta_0,\delta_z/2\right\rbrace.
    	$$
    	This proves the claim.
    \end{proof}

    Since $m_n\rightarrow+\infty$ as $n\rightarrow\infty$, this claim implies
    $$
    \lim_{n\rightarrow\infty}
    \frac{d_{\Fc}\big(f^{m_nn_0}(x_n),f^{m_nn_0}(w_n)\big)}{d_{\Fc}\big(f^{m_nn_0}(z),p\big)}\longrightarrow+\infty
    $$
	as $n\rightarrow\infty$. 
	
    However, $d_{\Fs}(p,f^{m_nn_0}(x_n))\leq 1$, this contradicts that the holonomy maps stable foliation in a center-stable leaf is $C^{1+}$, see Theorem B of \cite{PSW}. This proves
	$$
	\lam^c(p)=\lam^c(q), \qquad \forall p,q\in\Per(f|_{\Lam}).
	$$
\end{proof}

\begin{cor} \label{cor:unifhyp}
	$\Lam$ is a uniformly hyperbolic attractor.
\end{cor}

\begin{proof}
	For every $p\in\Per(f|_{\Lam})$, we have $p=p_+$ or $p=p_-$. If $p=p_+$, then $\Fc_+(p)$ is contained in the stable manifold of $p$. If $p=p_-$, then $\Fc_+(p)$ is contained in the stable manifold of $p$. From the global product structure of $\Fcs$ and $\Fu$, this implies every pair of periodic points $p,q\in\Lam$ are homoclinic related. Moreover, every periodic point $r\notin\Lam$ satisfies $r\neq r_+$ and $r\neq r_-$. This implies its stable manifold
	$$
	W^s(r)\subset\bigcup_{x\in(r_-,r_+)^c}\Fs(x).
	$$ 
	This implies
	$$
	W^s(r)\cap\Fus(r_-)=\emptyset, 
	\qquad {\rm and} \qquad
	W^s(r)\cap\Fus(r_+)=\emptyset.
	$$
	Thus $r$ is not homoclinic related to $r_+$ and $r_-$. For every  $p\in\Per(f|_{\Lam})$, its homoclinic class $H(p)\subseteq\Lam$.  Proposition \ref{prop:su-lamination} shows that $\Lam=\overline{\Per(f|_{\Lam})}$, this implies 
	$$
	\Lam=H(p), \qquad \forall p\in\Per(f|_{\Lam}).
	$$
	
    Proposition \ref{prop:samelyap} shows that $\lam^c(p)\leq\lam^c(A)<0$ for every $p\in\Per(f|_{\Lam})$. We apply the Main Theorem of \cite{bgy2009hyperbolicity}, which shows that $\Lam$ is a hyperbolic set. Moreover, the unstable manifold 
	$$
	W^u(x)=\Fu(x)\subset\Lam, 
	\qquad \forall x\in\Lam.
	$$
	This implies $\Lam$ is a hyperbolic attractor.
\end{proof}

The following proposition finishes the proof of Theorem \ref{thm:main}.

\begin{prop}\label{prop:whole-mfd}
    The hyperbolic attractor $\Lam$ is the whole of $\bbT^3$ and so $f$ is an Anosov diffeomorphism.
\end{prop}

\begin{proof}
	We have $\Lam=\bbT^3$ if $x=h\inv(h(x))$ for every $x\in\bbT^3$. Assume there exists $x\in\bbT^3$ such that $x\neq h\inv(h(x))$. Then for every $y\in\Aus(h(x))$, $h\inv(y)$ is a non-trivial center segment.
	
	\begin{claim}
		There exists $z'\in\Aus(h(x))$, such that
		the set
		$$
		B^{us}_{1}(z')=\left\lbrace y'\in\Aus(z'):d_{\Aus}(y',z')\leq 1\right\rbrace 
		$$
		satisfies
		$$
		A^{-k}\left(B^{us}_{1}(z')\right)\cap B^{us}_{1}(z')=\emptyset,
		\qquad \forall k>0.
		$$
	\end{claim}
    \begin{proof}[Proof of the claim]
    	There are two possibilities, either $\Aus(h(x))$ contains no periodic points of $A$, or $\Aus(h(x))$ contains a periodic point.
    	
    	If $\Aus(h(x))$ contains no periodic points of $A$, then
    	$$
    	A^{-k}(\Aus(h(x)))\cap\Aus(h(x))=\emptyset,
    	\qquad \forall k>0
    	$$ 
    	Since $B^{us}_{1}(h(x))\subset\Aus(h(x))$, we only need to choose $z'=h(x)$.
    	
    	If $\Aus(h(x))$ contains a periodic point $p$ with period $\pi$, then $A^{\pi}:\Aus(h(x))\rightarrow \Aus(h(x))$ is a linear Anosov action on the plane. So there exists $z'\in\As(p)\setminus\{p\}\subset\Aus(x)$ sufficiently far from $p$ in $\As(p)$, which satisfies
    	$$
    	A^{-k\pi}\left(B^{us}_{1}(z')\right)\cap B^{us}_{1}(z')=\emptyset,
    	\qquad \forall k>0.
    	$$  
    	This implies
    	$A^{-k}\left(B^{us}_{1}(z')\right)\cap B^{us}_{1}(z')=\emptyset$ for every $k>0$
    \end{proof}
    
    Let $z$ be the center point of the non-trivial segment $h^{-1}(z')$, and $\delta>0$ satisfying 
    $$
    B_{3\delta}(z)\cap\Lam=\emptyset.
    $$ 
    The semiconjugacy $h:\bbT^3\rightarrow\bbT^3$ is a continuous map. 
    There exists $\ep_0>0$, such that if $d(x_1,x_2)<\ep_0$ and $h(x_2)\in\cA^{us}(h(x_1))$, then $d_{us}(h(x_1),h(x_2))<1$. Here $d_{us}(\cdot,\cdot)$ is the distance in each leaf of $\cA^{us}$.
    
	Since $\Lam$ is a hyperbolic attractor, there exists a constant $0<\ep<\min\{\delta,\ep_0\}$, such that the $\ep$-neighborhood $B_{\ep}(\Lam)$ satisfies
	$$
	f^k\left(  B_{\ep}(\Lam) \right) \subset B_{\delta}(\Lam),
	\qquad \forall k>0.
	$$
	This implies
	$$
	f^{-k}\left( B_{\delta}(z) \right)\cap B_{\ep}(\Lam)=\emptyset,
	\qquad \forall k>0.
	$$
	Otherwise, we have $B_{\delta}(z)\cap B_{\delta}(\Lam)\neq\emptyset$, which contradicts to $B_{3\delta}(z)\cap\Lam=\emptyset$. 
	Thus for every $k>0$, we have
	$$
	B_{\ep}(f^{-k}(z))\cap\Lam=\emptyset, 
	\qquad {\rm and} \qquad
	h\left( B_{\ep}(f^{-k}(z)) \right)\subset \cA^{us}(A^{-k}(z')).
	$$

	Moveover, since $0<\ep<\ep_0$, for every $y\in B_{\ep}(f^{-k}(z))$, we have 
	$$
	d_{us}\left(h(y),A^{-k}(z')\right)<1.
	$$
    This implies
	$$
	h\left( B_{\ep}(f^{-k}(z)) \right)\subset B^{us}_1(A^{-k}(z')),
	\qquad \forall k>0.
	$$
	Since $A^{-k}\left(B^{us}_{1}(z')\right)\cap B^{us}_{1}(z')=\emptyset$ for every $k>0$, the semiconjugacy property implies the sequence of balls 
	$$
	\left\lbrace B_{\ep}(f^{-k}(z)): k>0 \right\rbrace 
	$$
	are mutually disjoint. This is absurd since the volume of each $B_{\ep}(f^{-k}(z))$ has a lower bound. Thus $h$ is injective everywhere and $\Lam=\bbT^3$.
\end{proof}

\bigskip

\noindent\textbf{Acknowledgements.}
This research was partially funded by the Australian Research Council.
Y. Shi is supported by NSFC 11701015, 11831001.



\bibliographystyle{alpha}
\bibliography{dynamics}

\end{document}